\documentclass[12pt]{amsart}
\usepackage{mathrsfs,amsmath,amssymb,amsfonts,amsthm,latexsym,marginnote}
\theoremstyle{plain}
\newtheorem{theorem}{Theorem}[section]
\newtheorem{proposition}[theorem]{Proposition}
\newtheorem{corollary}[theorem]{Corollary}

\theoremstyle{definition}

\newtheorem{example}[theorem]{Example}

\date{\today}
\author[F.~Xanthos]{Foivos Xanthos}
\address{Department of Mathematical and Statistical Sciences, University of Alberta, Edmonton, AB, Canada T6G\,2G1}
\email{foivos@ualberta.ca}

\begin{document}

\title{A version of Kalton's theorem for the space of regular operators}\thanks{This research was supported by NSERC}
\maketitle

\begin{abstract}In this note we extend some recent results in the space of regular operators. In particular, we provide the following Banach lattice version of a classical result of Kalton:
Let $E$ be an atomic Banach lattice with an order continuous norm and $F$ a Banach lattice. Then the following are equivalent:
(i) $L^r(E,F)$ contains no copy of $\ell_\infty$, \,\, (ii) $L^r(E,F)$ contains no copy of $c_0$, \,\, (iii) $K^r(E,F)$ contains no copy of $c_0$, \,\, (iv) $K^r(E,F)$ is a (projection) band in $L^r(E,F)$, \,\, (v) $K^r(E,F)=L^r(E,F)$.
\end{abstract}

\section{Introduction and Notations}
A trademark in the geometry of the space of bounded operators between Banach spaces is the paper \cite{KALT74} of N.J. Kalton. In (\cite{KALT74}, Theorem 6) it is discussed the embeddability of $c_0$ and $\ell_\infty$ into the space of bounded operators and the space of compact operators. In \cite{BU12b}, the authors motivated by the  work of Kalton studied the embeddability of $c_0$ and $\ell_\infty$ in the space of regular operators. In this study it is assumed that the range space is an arbitrary Banach lattice and the domain space is an Orlicz sequence space $\ell_\phi$, where $\phi$ and $\phi^*$ satisfies the $\Delta_2$ condition(in this case, $\ell_\phi$ is reflexive). The purpose of this paper is to extend this study to the case where the domain space is an atomic Banach lattice with an order continuous norm. Our main result also improves some results in \cite{BU11} and \cite{BU12}.

In certain cases, the space of regular operators can be represented as a Banach lattice-valued sequence space. In \cite{BU12b}, the authors derived their results by using techniques from this class of spaces. This approach has been successfully applied by Qingying Bu and coauthors to study several problems in the space of regular operators and positive tensor products of Banach lattices(see \cite{BU12,BU11,BUSKES09,BU09,BU08,BU12b}). In this paper we use a different approach and we work directly on the space of regular operators. In particular by using the concept of cone isomorphism(see \cite{POLY11,SCHE13}) we are able to modify Kalton's main argument, in order  to work for the regular operator norm. Moreover in our approach we use ideas and results from the work of Chen and Wickstead(\cite{Chen06,ChenWick07,ChenWick98,Wick07,Wick96}) in the space of regular operators.

In this paper, Banach spaces are always assumed to be infinite dimensional. Let $E$ and $F$ be Banach lattices, the space of bounded(compact) operators from $E$ to $F$ is denoted by $L(E,F)$($K(E,F)$) and the operator norm is denoted by $||\cdot||$.  The space of regular operators from $E$ into $F$(i.e. the space generated by the positive operators form $E$ into $F$)  is denoted by $L^r(E,F)$ and with $K^r(E,F)$ we denote the space of regular compact operators from $E$ into $F$. In general these spaces are not complete under the operator norm. However, there exist natural norms for which they are complete. In $L^r(E,F)$ this norm is called regular norm, is denoted by $|| \cdot ||_r$ and is defined as follows $||T||_r=\inf\{||S|| \,\, | \,\, S \geq \pm T\}$ for each $T \in L^r(E,F)$. In $K^r(E,F)$ the corresponding norm is called k-norm, is denoted by $|| \cdot ||_k$ and is defined as follows $||T||_k=\inf\{||S|| \,\, | \,\, S \in K^r(E,F), S \geq \pm T\}$ for each $T \in K^r(E,F)$. Apart from the obvious relation $||\cdot||_k\geq ||\cdot||_r \geq ||\cdot||$  there are not many good connections about these norms. In particular, in many cases $K^r(E,F)$ fails to be complete under $||\cdot||_r$. However, it is important for our study that these three norms coincide in the cone of positive operators. At the following when we do not specify a particular norm, we assume that the aforementioned operator spaces are equipped with their natural norm. Finally, we recall here a classical result in this topic: If $F$ is Dedekind complete then $L^r(E,F)$ is a Dedekind complete Banach lattice under the regular norm and $||T||_r=|| \, \, |T| \,\,||$(see \cite{ali}, Theorem 15.2).

\section{The main result}

In \cite{SCHE13}, the author introduced a type of "cone isomorphism", which we will use in the sequel. Let $X$ be a Banach lattice and $Y$ a Banach space, an operator $T: X \rightarrow Y$ is a \textbf{0-cone isomorphism} if there exist constants $A$ and $B$ such that $A||x|| \leq ||Tx|| \leq B||x||$ for each $x \in X_+$. Note that in this case we have that $||Tx|| \leq B(||x^+||+||x^-||) \leq 2B||x||$ for each $x \in X$, hence $T$ is automatically a bounded operator. In \cite{POLY11}, the authors proved that if there exists a stronger type of "cone isomorphism" of $c_0$ into a Banach space $Y$, then $c_0$ is embeddable in $Y$. Below we show that the same holds if we assume the existence of a \textbf{0-cone isomorphism}.

\begin{proposition}\label{c0_weak}
If there exists a 0-cone isomorphism from $c_0$ into a Banach space $Y$, then $c_0$ is embeddable in $Y$.
\end{proposition}

\begin{proof}
Suppose that $T:c_0 \rightarrow Y$ is a 0-cone isomorphism and $A,B$ constants such that $A||x|| \leq ||Tx|| \leq B||x||$ for each $x \in c_0^+$. Let $y_n=T(e_n)$ then we have that $\inf{||y_n||}>0$ and $||\sum_{i=1}^n a_i y_i||\leq B \max\{a_i | i=1,..,n\}$ for each $a_i \geq 0$ and $n \in \mathbb{N}$. This implies that $||\sum_{i=1}^n a_i y_i||\leq 2B \max\{|a_i| \,\, | \,\, i=1,..,n\}$ for each $a_i \in \mathbb{R}$ and $n \in \mathbb{N}$(see \cite{POLY11}, p.682). Since $e_n \xrightarrow{{w}} 0$ and $T$ is bounded, we have that $(y_n)$ is weakly null  and  by passing to a subsequence we can assume that $(y_n)$ contains a basic sequence. Therefore $(y_n)$ is equivalent to the standard unit vector basis of $c_0$(see \cite{MEGG98}, Theorem 4.3.7).
\end{proof}

We say that a Banach lattice $X$ is \textbf{positively embedded} in an ordered Banach space $Y$ if there exist a positive isomorphism from $X$ into $Y$.

\begin{proposition}\label{ro_norm}
Suppose that $E,F$  and $X$ are Banach lattices.
\begin{enumerate}
\item[(i)]If $X$ is positively embedded in $(L(E,F),|| \cdot ||)$, then $X$ is positively embedded in $(L^r(E,F),|| \cdot ||_r)$
\item[(ii)]If $X$ is positively embedded in $(K(E,F),|| \cdot ||)$, then $X$ is positively embedded in $(K^r(E,F),|| \cdot ||_k)$
\end{enumerate}
\end{proposition}

\begin{proof}
$(i)$Suppose that $T: X \rightarrow L(E,F)$ is a positive embedding of $X$ into $L(E,F)$, and $A||x|| \leq ||T(x)|| \leq B||x||$ for each $x \in X$. Let $x \in X$ then $T(x)=T(x^+)-T(x^-)$, thus $Range(T) \subseteq L^r(E,F)$. Moreover, since $\pm T(x) \leq T(|x|)$ we have that $$||T(x)||_r \leq ||T(|x|)|| \leq B||x||.$$ Hence $A||x|| \leq ||T(x)||_r \leq B ||x||$, since $||T(x)|| \leq ||T(x)||_r$ and $X$ is embeddable in $(L^r(E,F),|| \cdot ||_r)$. The above proof works for argument $(ii)$ as well.

\end{proof}

\begin{proposition}\label{ro_norm2}
Suppose that $E,F$  and $X$ are Banach lattices.
\begin{enumerate}
\item[(i)]If $T$ is a positive isomorphism from $X$ into $(L^r(E,F),|| \cdot ||_r)$, then $T$ is a 0-cone isomorphism from $X$ into  $(L(E,F),|| \cdot ||)$.
\item[(ii)]If $T$ is a positive isomorphism from $X$ into $(K^r(E,F),|| \cdot ||_k)$, then $T$ is a 0-cone isomorphism from $X$ into  $(L^r(E,F),|| \cdot ||_r)$.
\end{enumerate}
\end{proposition}

\begin{proof}
$(i)$Suppose that $T: X \rightarrow L^r(E,F)$ is a positive embedding of $X$ into $L^r(E,F)$, and $A||x|| \leq ||T(x)||_r \leq B||x||$ for each $x \in X$, then for each $x \in X_+$ we have that $A||x|| \leq ||T(x)||\leq B||x||$ and thus $T$ is a 0-cone isomorphism from $X$ into  $(L(E,F),|| \cdot ||)$. The above proof works for argument $(ii)$ as well.

\end{proof}

\begin{proposition}\label{emb}
Let $E$ and $F$ be Banach lattices. Then $E^*$ and $F$ are positively embedded in $(K^r(E,F),||\cdot||_k)$ and $(L^r(E,F),||\cdot||_r)$
\end{proposition}

\begin{proof}
We will prove that $E^*$ and $F$ are positively embedded in $K(E,F)$. Then the conclusion follows from Proposition \ref{ro_norm}. Let $y_0 \in F_+$ with $||y_0||=1$, we define the following operator $\Phi: E^* \rightarrow K(E,F)$, $\Phi(x^*)=x^* \otimes y_0$, clearly $\Phi$ is positive and it is an isometry of $E^*$ onto a subspace of $K(E,F)$. Indeed, $||\Phi(x^*)||=\sup\{||x^*(x)y_0|| \,\, | \,\, x \in B_E\}=||y_0||\sup\{|x^*(x)| \,\, | \,\, x \in B_E\}=||x^*||$, thus $E^*$ is positively embedded in $K(E,F)$. The proof for $F$ is analogous. Let $x^*_0 \in E^*_+$ with $||x^*_0||=1$ and define the following operator $\Psi: F \rightarrow K(E,F)$, $\Psi(y)=x_0^* \otimes y$, clearly $\Psi$ is positive and it is an isometry of $F$ onto a subspace of $K(E,F)$. Indeed, $||\Psi(y)||=\sup\{||x^*_0(x)y|| \,\, | x \in B_E\}=||y||\sup\{|x^*_0(x)| \,\, | x \in B_E\}=||y||$, thus $F$ is positively embedded in $K(E,F)$.
\end{proof}

\begin{proposition}\label{wick1}Let $E$ and $F$ be Banach lattices. If $c_0$ is not positively embedded in $(K^r(E,F), || \cdot ||_k)$ then  $(K^r(E,F), || \cdot ||_k)$ is a KB-space and $E^*,F$ are KB-spaces.
\end{proposition}

\begin{proof}
We will prove first, that $E^*$ and $F$ are KB-spaces. Suppose that $F$ is not a KB-space, then $F$ contains a positive basic sequence $y_n$ equivalent to the standard unit vector basis of $c_0$. Thus by Proposition \ref{emb} we have that $c_0$ is positively embedded in $K^r(E,F)$ a contradiction. Similarly we can prove that $E^*$ is a KB-space.

Next we will prove that  $K^r(E,F)$ is a sublattice of $L^r(E,F)$. Let $T \in K^r(E,F)$ and suppose that $T=T_1-T_2$, where $T_1,T_2 \in K_+(E,F)$ then $T_1 \geq T,0$ hence $0 \leq T^+ \leq T_1$ and by the Dodds-Fremlin domination theorem(see \cite{ali}, Theorem 16.20) we have that $T^+ \in K^r(E,F)$. Next we will verify that $|| \cdot ||_k$ and $|| \cdot ||_r$ coincide in $K^r(E,F)$. Let $T \in K^r(E,F)$ then by definition we have that $||T||_k \geq ||T||_r$. Moreover $|T| \in K^r(E,F)$, so $||T||_r=||\, |T| \, ||  \geq ||T||_k$. Hence $K^r(E,F)$ is a Banach lattice under $|| \cdot ||_k$ and since $c_0$ is not positively embedded, we have that $c_0$ is not lattice embeddable and $(K^r(E,F), || \cdot ||_k)$ is a KB-space.

\end{proof}

\begin{corollary}\label{c0_inL(E,F)}
Let $E$ and $F$ be Banach lattices. If $c_0$ is embeddable in $(K^r(E,F), || \cdot ||_k)$, then $c_0$ is embeddable in $(L^r(E,F), || \cdot ||_r)$
\end{corollary}

\begin{proof}
In view of Proposition \ref{wick1}, we can assume that $c_0$ is positive embeddable in $K^r(E,F)$, then there exists a positive isomorphism from $c_0$ into $K^r(E,F)$, thus by  Proposition \ref{ro_norm2} and Proposition \ref{c0_weak} we have that $c_0$ is embeddable in $L^r(E,F)$.

\end{proof}

We say that a Banach lattice $E$ has the \textbf{positive dual Schur property(PDSP)} if every positive $w^*$-null sequence converges in norm to zero. This property was introduced in \cite{AQZZ11} and further developed in \cite{WNUK13}. We remark here that every AM-space with a strong unit has the (PDSP) and every space with an order continuous norm fails the (PDSP)(see in \cite{WNUK13}).

\begin{theorem}\label{PDSP_prop}
Let $E$ and $F$ be Banach lattices, such that $E$ fails the (PDSP), then
\begin{itemize}
\item[(i)] $c_0$ is embeddable in $(L^r(E,F),||\cdot||_r)$ if and only if $\ell_\infty$ is embeddable in $(L^r(E,F),||\cdot||_r)$,
\item[(ii)] If the supremum of every increasing sequence of positive compact operators is a compact operator, then $E^*$ and $F$ are KB-spaces.
\end{itemize}
\end{theorem}

\begin{proof}
(i)Suppose that $c_0$ is embeddable in $L^r(E,F)$ and $\ell_\infty$ is not  embeddable in $L^r(E,F)$. We will prove first that $F$ is a KB-space. Suppose that $c_0$ is lattice embeddable in $F$ and $(y_n) \subseteq F_+$ is equivalent to the standard unit vector basis of $c_0$. Let $(x^*_n) \subseteq E^*_+$ such that $||x^*_n||=1$ and $x^*_n \xrightarrow{{w^*}} 0$. We define the following map, $$\Phi: \ell_\infty \rightarrow L(E,F), \Phi(\xi)(x)=\sum_{n=1}^\infty \xi_n x_n^*(x)y_n \,\, \text{for each} \,\, x \in E.$$ Fix some $\xi \in \ell_\infty^+$, then clearly $\Phi(\xi) \in L_+(E,F)$ and is not hard to check that  $\Phi$ is an isomorphism from $\ell_\infty$ into $L(E,F)$. Thus by Proposition \ref{ro_norm} we have that $\ell_\infty$ is embeddable in $(L^r(E,F),|| \cdot|| _r)$, a contradiction.

Since $F$ is a KB-space, $L^r(E,F)$ is a Banach lattice and $c_0$ is lattice embeddable in $L^r(E,F)$. Hence, there exists a sequence $(S_n) \subseteq L_+(E,F)$ equivalent to the standard basis of $c_0$.  For each $\xi=(\xi_n) \in \ell_\infty$ and $x \in E$ we have that $\sum_{n=1}^\infty \xi_nS_nx$ is weakly unconditionally Cauchy in $F$. Indeed, for each $x \in B_E$, each finite  $J \subseteq \mathbb{N}$ and $\epsilon_n=\pm 1$ we have that $||\sum_{i \in J} \epsilon_i \xi_i S_ix|| \leq || \sum_{i \in J} \epsilon_i \xi_i S_i|| \leq || \sum_{i \in J} \epsilon_i \xi_i S_i||_r \leq ||\xi|| $, up to a constant. The conclusion follows from (\cite{kalton}, Lemma 2.4.6).  Since $F$ is a KB-space we have that $\sum_{n=1}^\infty \xi_nS_nx$ converges unconditionally in $F$(see \cite{kalton}, Theorem 2.4.11). Thus, the operator $\Psi: \ell_\infty \rightarrow L^r(E,F)$ given by $\Psi(\xi)=\sum_{n=1}^\infty \xi_n S_n$ in the strong operator topology is well defined. $\Psi$ is a bounded operator, since it is positive. In particular,  since $\ell_\infty$ is not embeddable in $(L^r(E,F),||\cdot||_r)$, it follows by (\cite{diestel}, Theorem 10, p. 156) that $\Psi$ is weakly compact and by (\cite{ali}, Theorem 19.6), we have that $\Psi$ is a Dunford Pettis operator. Thus $\lim_{n \rightarrow \infty}|| \Psi(e_n)||=\lim_{n \rightarrow \infty}|| S_n||=0$, which contradicts the fact that $S_n$ is equivalent to the standard unit vector basis of $c_0$.

(ii)Suppose that $E^*$ does not have order continuous norm, then $\ell_1$ is embeddable in $E$, as the range of a positive projection $P$(see \cite{mn1991}, Proposition 2.3.11 and Theorem 2.4.14). We denote with $(z_n) \subseteq E$ the corresponding sequence that is equivalent with the standard unit vector basis of $\ell_1$ and $(z^*_n)$ the biorthogonal functionals of $(z_n)$.  Let $y_n$ a positive bounded sequence in $F$ with no convergent subsequence, then we define the following operators $$T_n(x)=\sum_{i=1}^n z_i^*(P(x)) y_i, S(x)=\lim T_n(x) \,\, \text{for each } \ x \in E.  $$

The above limit exists, since $\sum_{i=1}^\infty |z_i^*(P(x))|<\infty$ and $(y_n)$ is bounded. $(T_n)$ is a sequence of positive compact operators such that $T_n \uparrow S$.  The operator $S$  is not compact since  $S(z_n)=y_n$ which has no convergent subsequence, a contradiction.

Suppose that $F$ is not a KB-space, then there exists a sequence $(y_n)$  in $F_+$ equivalent to the standard unit vector basis of $c_0$.  Let $z_n$ a normalized,  $w^*$-null sequence in $E^*_+$. We define, the following operators
$$T_n(x)=\sum_{i=1}^n z_i(x)y_i, S(x)=\lim T_n(x) \,\, \text{for each } \ x \in E.$$

The above limit exists since $(z_i(x))_i \in c_0$ and $y_n$ is equivalent to the standard unit vector basis of $c_0$. Then $(T_n)$ is a sequence of positive compact operators such that $T_n \uparrow S$.
At the following we will prove that $S$ is not compact, by proving that $S^*$ is not compact. Let $(y_n^*) \subseteq F^*$ the biorthogonal functionals of $(y_n)$, then $(y^*_n)$ is a bounded sequence and we have that $S^*(y_m^*)(x)=y_m^*(S(x))=\sum_{n=1}^\infty z_n(x)y_m^*(y_n)=z_m(x)$ for each $x \in E$. Hence $S^*(y^*_m)=z_m$, which does not have a norm convergent subsequence, a contradiction.
\end{proof}

The following example illustrates that in the previous Theorem the assumption that $E$ does not have the (PDSP) cannot be dropped. It is interesting to observe that in the setting of the space of bounded operators of Banach spaces, the situation is different. In particular, the corresponding argument $(i)$ of Theorem \ref{PDSP_prop} is valid for any Banach space $E$ and $F$. That is,  $c_0$ is embeddable in $L(E,F)$ iff $\ell_\infty$ is embeddable in $L(E,F)$(see \cite{FEDER80}, p. 201).

\begin{example}\label{exmp1}
Let $E$ be an AM-space with a strong unit and $F$ an order continuous Banach lattice, which is not a KB-space. According to Proposition \ref{emb} we have that $c_0$ is embeddable in $L^r(E,F)$. By (\cite{Chen06}, Proposition 3.1)  we have that $L^r(E,F)$ has an order continuous norm, thus $\ell_\infty$ is not embeddable in $L^r(E,F)$. Moreover $K^r(E,F)$ is a band in $L^r(E,F)$ according to (\cite{ChenWick07}, Theorem 3.7).
\end{example}

Before we proceed to our main result, we remark here that if $E$ is an atomic Banach lattice with an order continuous norm and $F$ is any Banach lattice, then $L^r(E,F)$ is always a vector lattice(see \cite{Wick07}, Theorem 3.4). However, the space of regular compact operators may fail to be a vector lattice and moreover  $|| \cdot ||_r$ and $|| \cdot ||_k$ may not be equivalent in the space of regular compact operators(see \cite{ChenWick98}, Corollary 5.4). The following extends part (b) of (\cite{BU12b}, Theorem 9) and the KB-part of (\cite{BU11}, Theorem 8)

\begin{theorem}\label{main_thm}Let $E$ be an atomic Banach lattice with an order continuous norm. If $F$ is any Banach lattice, the following are equivalent.

\begin{enumerate}
\item[(i)]$(L^r(E,F),||\cdot||_r)$ contains no copy of $\ell_\infty$
\item[(ii)]$(L^r(E,F),||\cdot||_r)$ contains no copy of $c_0$
\item[(iii)]$(K^r(E,F),||\cdot||_k)$ contains no copy of $c_0$
\item[(iv)]$K^r(E,F)$ is a (projection) band in $L^r(E,F)$
\item[(v)]$K^r(E,F)=L^r(E,F)$
\end{enumerate}

\end{theorem}

\begin{proof}
Since $E$ has an order continuous norm, $E$ fails the positive dual Schur property. Hence by Theorem \ref{PDSP_prop} we have that (i) $\Leftrightarrow$ (ii).

(ii) $\Rightarrow$ (iii) It follows by Corollary \ref{c0_inL(E,F)}.

(iii) $\Rightarrow$ (iv)By Proposition \ref{wick1}, $K^r(E,F)$ is a KB-space and $E^*,F$ are KB-spaces. By the Dodds Fremlin domination theorem, it follows that $K^r(E,F)$ is an ideal in $L^r(E,F)$. Let $(T_a) \subseteq K_+(E,F)$ such that $T_a \uparrow S$ in $L^r(E,F)$, then $||T_a||_k=||T_a||_r\leq ||S||_r$ for each $a$, thus since $(K^r(E,F),||\cdot||_k)$ is a KB-space we have that $T_a$ is $||\cdot||_k$ convergent to some $T \in K_+(E,F)$. In particular $T=\lim T_a$ with respect to $||\cdot||_r$ and thus $T=\sup\{T_a\}=S$ and $K^r(E,F)$ is  a projection band in $L^r(E,F)$.

(iv) $\Rightarrow$ (v)
Suppose that $\{e_i \,\, | \,\, i \in I\}$ is a complete disjoint sequence of atoms of $E$. Let $\delta \subseteq I$ finite, then it is easy to verify that the subspace $[e_i \,\, | i \in \delta]$ is band in $E$. Let $P_\delta$ be the corresponding band projection. For each $x \in E_+$ we have that $0 \leq P_\delta(x) \leq x$ and that the net $(P_\delta(x))_\delta$ with respect to the inclusion ordering is increasing. Since $E$ has order continuous norm the interval $[0,x]$ is norm compact and thus $(P_\delta(x))_\delta$ is norm convergent. In particular we have that $x=\lim P_\delta(x)$. Indeed, $|x-\lim P_\delta(x)| \wedge e_i=\lim |(I-P_\delta)x|\wedge e_i=0$ for each $i \in I$, hence $x=\lim P_\delta(x)$. Let $T\geq 0$ and define $S_\delta(x)=TP_\delta(x)$ for each $x \in E$. Then $S_\delta$ is a positive compact operator since $P_\delta$ is a finite rank operator and $T(x)=\lim S_\delta(x)$ for each $x \in E_+$. Thus we have that $S_\delta \uparrow T$ and by (iv) we have that $T$ is compact.

(v) $\Rightarrow$ (i)
By Theorem \ref{PDSP_prop}, $E^*$ and $F$ are KB-spaces. By (\cite{ChenWick07}, Theorem 2.8), $K^r(E,F)$ has an order continuous norm. Therefore since $K^r(E,F)=L^r(E,F)$, we have that $\ell_\infty$ is not embeddable in $L^r(E,F)$.

\end{proof}

Let $E$ and $F$ be Banach lattices. The Fremlin tensor product of $E$ and $F$ is denoted $E\hat{\otimes}_{|\pi|} F$ and is defined as the completion of $E \otimes F$ under the positive projection norm $|| \cdot ||_{|\pi|}$(see \cite{FREM74}, Definition 1C, p.88). The Fremlin tensor product of $E$ and $F$ is a Banach lattice under the natural order and the positive projection norm(see \cite{FREM74}, Theorem 1E, p.89). We recall here the following standard representation of the dual of $E\hat{\otimes}_{|\pi|} F$(for a proof see also the proof of Proposition 2 in \cite{BU13b})

\begin{theorem}\label{Fremlin}(\cite{SCHA80}, Theorem 3.2, p.204)

Let $E,F$ be Banach lattices, then $(L^r(E,F^*), || \cdot ||_r)$ is lattice isometric with $(E\hat{\otimes}_{|\pi|} F)^*$
\end{theorem}

Using Theorem \ref{main_thm} and the above representation we can extend Theorem 6.8(ii) in \cite{BU12}, where the authors stated the following result in the case where $E$ is a reflexive atomic Banach lattice.

\begin{theorem}\label{tensor}Let $E$ be an atomic Banach lattice with an order continuous norm and $F$ any  Banach lattice. Then the following are equivalent \begin{enumerate}
\item[(i)] $\ell_1$ is not lattice embeddable in $E \hat{\otimes}_{|\pi|} F$
\item[(ii)] $K^r(E,F^*)=L^r(E,F^*)$ (and $\ell_1$ is not lattice embeddable in $F$ and $E$).
\end{enumerate}
\end{theorem}

\begin{proof}
According to the representation Theorem \ref{Fremlin} and standard results about the embeddability of $\ell_\infty$ in Banach lattices(see \cite{ali}, Theorem 14.21) we have that $\ell_1$ is not lattice embeddable in $E \hat{\otimes}_{|\pi|} F$ if and only if $\ell_\infty$ is not embeddable in $L^r(E,F^*)$. Hence the result follows by Theorem \ref{main_thm}. Moreover, note that if $\ell_\infty$ is not embeddable in $L^r(E,F^*)$, then $\ell_\infty$ is not embeddable in $F^*$ and $E^*$  by Proposition \ref{emb}, thus $\ell_1$ is not lattice embeddable in $F$ and $E$.
\end{proof}

In the above result, the assumption of the order continuity of $E$ or the atomicity of $E$ standalone, cannot yield the equivalence of (i),(ii). To illustrate this, we will use the following result:

\begin{theorem}\label{Chen_them}(\cite{Chen06}, Theorem 3.3)
Let $E,F$ be Banach lattices. If $F$ has the positive Schur property, then the regular norm $||\cdot||_r$ on $L^r(E,F)$ is order continuous if and only if $E^*$ has an order continuous norm.
\end{theorem}

\begin{example}\label{examp_last}Consider the following cases:
\begin{enumerate}
\item[(i)] $E$ is a non atomic $L_p$-space for some $1 < p < \infty$ and $F^*$ is a non atomic AL-space,
\item[(ii)] $E=\ell_\infty$ and $F^*=(\ell_\infty)^*$.
\end{enumerate}
Theorem \ref{Chen_them}, implies that in both cases $\ell_\infty$ is not embeddable in $L^r(E,F^*)$, hence $\ell_1$ is not lattice embeddable in $E \hat{\otimes}_{|\pi|} F$. Moreover we have that $K^r(E,F^*) \neq L^r(E,F^*)$ in both cases. Case (i) follows by (\cite{ChenWick98}, Theorem 4.9 ) and case (ii) follows by (\cite{WNUK13}, Proposition 2.8). Note also that in case (ii) we also have that $K^r(E,F)$ is a band in $L^r(E,F)$ by (\cite{ChenWick07}, Theorem 3.7).
\end{example}

\textbf{Acknowledgements}. The results were presented at the working seminar of the functional analysis group of the University of Alberta  and I would like to thank the participants for the useful discussions.

\end{document}